\documentclass[11pt,reqno]{amsart}
\usepackage[top=1.2in, bottom=1.2in, left=1.2in, right=1.2in]{geometry}
\usepackage{amsfonts, amssymb, amsmath, mathtools, dsfont}
\usepackage{bold-extra}
\usepackage[linktoc=page, colorlinks, linkcolor=blue, citecolor=blue]{hyperref}
\usepackage{tikz}

\numberwithin{equation}{section}

\makeatletter
\renewcommand{\part}{\@startsection{part}{0}%
	\z@{\linespacing\@plus\linespacing}{.5\linespacing}%
	{\normalfont\bfseries\scshape\raggedright}}
\makeatother

\DeclareMathOperator{\E}{\mathbb{E}} \DeclareMathOperator{\Var}{Var}    

\DeclarePairedDelimiter \abs{\lvert}{\rvert}  \DeclarePairedDelimiterX \ip[2]{\langle}{\rangle}{#1,#2} \DeclarePairedDelimiterXPP \Prob[1]{\mathbb{P}}\{\}{}{
   
   #1} \DeclarePairedDelimiterXPP \Probevent[1]{\mathbb{P}}(){}{
   
   #1}

\def \C {\mathbb{C}}

\def \R {\mathbb{R}}

\def \e {\varepsilon}
\def \d {\delta}

\def \s {\sigma}

\def \one {{\mathbf 1}}

\newtheorem{theorem}{Theorem}[section]

\newtheorem{lemma}[theorem]{Lemma}

\newtheorem{definition}[theorem]{Definition}

\theoremstyle{remark}
\newtheorem{remark}[theorem]{Remark}

\title{Two Friendly Proofs of the Berry--Esseen Theorem}

\author{Roman Vershynin}
\address{Department of Mathematics, University of California, Irvine, U.S.A.}
\email{rvershyn@uci.edu}

\thanks{Partially supported by the NSF Grant DMS 2451011 and the U.S. Air Force Grant FA9550-25-1-0294.}

\begin{document}

\begin{abstract}
	A gem of classical probability, the Berry--Esseen theorem provides a non-asymp\-totic form of the central limit theorem. This note gives a friendly and intuitive exposition of two different proofs of the Berry--Esseen theorem for nonidentically distributed random variables: the classical Fourier-analytic proof and a proof by Stein's method following E.~Bolthausen. Both proofs are self-contained; the reader may read either proof without reading the other. The exposition is suitable for a basic graduate course in probability.
\end{abstract}

\maketitle

\bigskip

The central limit theorem describes the eventual Gaussian behavior of sums of independent random variables. The Berry--Esseen theorem makes this approximation quantitative for every finite sum.

\begin{theorem}[Berry--Esseen theorem \cite{Berry, Esseen}]	\label{thm: berry-esseen}
	Let $X_{1}, \dots, X_{n}$ be independent mean-zero random variables whose sum
	$S_{n} \coloneqq X_{1} + \cdots + X_{n}$
	satisfies $\Var(S_{n}) = 1$. Let $G$ be a standard normal random variable.
	Then
	\[
	\sup_{a \in \R} \abs[\Big]{\Prob{S_{n} \le a} - \Prob{G \le a}}
	\le
	C \sum_{k=1}^{n} \E \abs{X_{k}}^{3},
	\]
	provided the right-hand side is finite, where $C$ is an absolute constant.
\end{theorem}

To see how this implies the classical central limit theorem, let $Y_1,Y_2,\ldots$ be independent mean-zero random variables with unit variances and uniformly bounded third moments. Applying Theorem~\ref{thm: berry-esseen} to $X_k=Y_k/\sqrt n$ gives
\begin{equation}	\label{eq: introduction BET}
	\sup_{a\in\R}
	\abs*{\Prob[\Big]{\frac{Y_1+\cdots+Y_n}{\sqrt n}\le a}-\Prob{G\le a}}
	=O\left(\frac1{\sqrt n}\right).
\end{equation}
Thus the normalized sums converge in distribution to $G$. More precisely, \eqref{eq: introduction BET} gives uniform convergence of the distribution functions, the optimal rate $O(n^{-1/2})$, and a bound valid for every finite $n$.

\medskip

We prove the theorem twice. Part~I gives the classical Fourier-analytic argument, including an intuitive proof of Esseen's smoothing inequality. Part~II gives a proof by Stein's method, following an extremely concise argument of Bolthausen \cite{Bolthausen}. The two parts are logically independent: neither proof invokes a definition, lemma, or estimate from the other. A reader interested in just one method may therefore go directly to the corresponding part.

There are several existing textbook proofs of the Berry--Esseen theorem. The classical Fourier-analytic proof is covered, for example, in Feller \cite[Ch.~XVI, Sec.~5, Thm.~2]{Feller}; for an introduction to Stein's method, see Ross \cite{Ross} and Chen, Goldstein, and Shao \cite{CGS}.

\part{The Fourier-analytic proof}

This part is self-contained. It begins with the proof strategy, recalls the needed Fourier analysis, and develops its own smoothing and characteristic-function estimates. 

\section{The proof strategy}

Here is a bird's-eye, non-rigorous sketch of the proof of Theorem~\ref{thm: berry-esseen}. 

\smallskip 

\subsection{Smoothing.}

We begin by writing
\begin{equation}	\label{fourier: eq: prob Sn}
	\Prob{S_{n} \le a} = \E \one_{\{ S_{n} \le a \}}.
\end{equation}
The difficulty is that the function \(x \mapsto \one_{\{x \le a\}}\) is too rough: it has a discontinuous drop from $1$ to $0$ at the point $a$. To mitigate this issue, we smooth the indicator function: approximate it with a function $f(x)$ that decreases from $1$ to $0$ gradually. With such approximation, \eqref{fourier: eq: prob Sn} becomes
\[
\Prob{S_{n} \le a} \approx \E f(S_{n}).
\]
A rigorous version of this smoothing step will be given in Lemma~\ref{fourier: lem: smoothing discrepancy}.

\subsection{Fourier transform.}

Next, we replace $f(x)$ by an even simpler function -- a complex exponential $e^{2 \pi i t x}$. This can be done using the Fourier transform inversion formula (recalled in \eqref{fourier: eq: Fourier inversion} below): 
\[
\E f(S_{n}) 
= \E \int_\R e^{2\pi i t S_{n}} \, \widehat{f}(t) \, dt
= \int_\R \E \left[ e^{2\pi i t S_{n}} \right] \, \widehat{f}(t) \, dt.
\]
Applying the same method to a standard normal random variable $G$, subtracting the two expressions, and using Jensen's inequality, we obtain
\begin{equation}	\label{fourier: eq: probs via chf}
	\abs{\Prob{S_{n} \le a} - \Prob{G \le a}}
	\lesssim \int_\R \abs*{\E e^{2\pi i t S_n} - \E e^{2\pi i t G}} \, \abs{\widehat{f}(t)} \, dt.
\end{equation}
This reduces the problem to comparing the {\em characteristic functions} $\E e^{i t S_n}$ and $\E e^{i t G}$. For the standard normal random variable $G$, a simple computation yields
\begin{equation}	\label{fourier: eq: normal chf}
	\E e^{i t G} = \exp(-t^2/2), \quad t \in \R.
\end{equation}

\subsection{Taylor approximation.}

For $S_n$, independence of the random variables $X_k$ gives
\[
\E e^{i t S_{n}}
= \E \left[ e^{i t X_{1}} \right] \cdots \E \left[e^{i t X_{n}} \right],
\]
which reduces the problem to computing the characteristic function $\E e^{itX_k}$ for each random variable $X_k$ separately. By assumption, 
$$
\E X_k = 0
\quad \text{and} \quad \E X_k^2 \eqqcolon \sigma_k^2 < \infty.
$$ 
Using the Taylor approximation $e^z \approx 1 + z + z^2/2$ and taking expectations, we get
\begin{align*}
	\E e^{itX_k}
	&\approx \E \bigl( 1 + itX_k - t^2 X_k^2 / 2 \bigr) \\
	&= 1 - \sigma_k^2 t^2 / 2
	\approx \exp \bigl( - \sigma_k^2 t^2 / 2 \bigr).
\end{align*}
Multiplying over $k$ gives
\begin{equation}	\label{fourier: eq: chf Sn}
	\E e^{itS_n}
	\approx \exp \bigl( - \sigma_1^2 t^2 / 2 \bigr)\cdots
  	\exp \bigl( - \sigma_n^2 t^2 / 2 \bigr)
	= \exp \bigl( - t^2 / 2 \bigr),
\end{equation}
since $\sigma_1^2 + \cdots + \sigma_n^2 = 1$ by assumption. Comparing to \eqref{fourier: eq: normal chf}, we see that 
$$
\E e^{itS_n} \approx \E e^{itG}.
$$
Substituting this approximation into \eqref{fourier: eq: probs via chf} completes the heuristic ``proof''.

\subsection{A challenge and a fix.}

The main issue with this heuristic argument lies in the quality of the Taylor approximation. We cannot expect 
\begin{equation}	\label{fourier: eq: chf Xk}
	\E e^{itX_k} \approx \exp(-\s_k^2 t^2/2)
\end{equation}
uniformly over all $t \in \R$. For example, if $X_k$ has the Rademacher distribution, its characteristic function equals $\cos(t)$, which does not even decay to zero as $\abs{t} \to \infty$.

There is an elegant fix for this issue: choose the smoothing function $f$ whose Fourier transform $\widehat{f}$ is supported on a {\em compact interval} centered at the origin. The integrand in \eqref{fourier: eq: probs via chf} would vanish outside that interval, and so it suffices to establish the approximation \eqref{fourier: eq: chf Xk} only for $t$ close to the origin. This idea leads to \emph{Esseen's smoothing inequality} (Theorem~\ref{fourier: thm: smoothing inequality}).

In Lemma~\ref{fourier: lem: chf rv}, we will prove \eqref{fourier: eq: chf Xk} in a neighborhood of the origin. Multiplying over $k$, we will deduce a rigorous form of \eqref{fourier: eq: chf Sn} in Lemma~\ref{fourier: lem: chf sum}. Plugging it into the smoothing inequality will complete the proof of Theorem~\ref{thm: berry-esseen}.

Now let's do this step by step.

\section{Background on Fourier transform}	\label{fourier: s: Fourier transform}

First, let's recall some basic facts about Fourier transform that we will use in the proof. 

The {\em Fourier transform} of an integrable function $f: \R \to \C$ is the bounded and continuous function $\widehat{f}: \R \to \C$ defined as
\begin{equation}	\label{fourier: eq: 	Fourier transform}
	\widehat{f}(t)
	= \int_\R e^{-2 \pi i t x} f(x) \, dx.
\end{equation}

If $f$ is a {\em Schwartz function} (that is, $f$ and all its derivatives decay faster than any polynomial at infinity), then $\widehat{f}$ is also a Schwartz function.

Any Schwartz function $f$ can be reconstructed from its Fourier transform using the {\em Fourier inversion formula}: 
\begin{equation}	\label{fourier: eq: Fourier inversion}
	f(x)
	= \int_\R e^{2 \pi i t x} \widehat{f}(t) \, dt.
\end{equation}
Comparing with \eqref{fourier: eq: Fourier transform}, we see that applying the Fourier transform twice yields $f(-x)$. In particular, the Fourier transform is an involution on even functions.

The Fourier transform preserves the $L^2$ norm: for a Schwartz function $f$, the {\em Plancherel identity} says that
\begin{equation}	\label{fourier: eq: Plancherel}
	\int_\R \abs{f(x)}^2 \, dx = \int_\R \abs{\widehat{f}(t)}^2 \, dt.
\end{equation}

The {\em convolution} of two integrable functions $f$ and $g$ is the integrable function $f \ast g$ defined as 
$$
(f \ast g)(x) = \int_\R f(y) g(x-y) \, dy.
$$ 
The {\em convolution theorem} states that 
\begin{equation}	\label{fourier: eq: convolution theorem}
	\widehat{f \ast g} = \widehat{f} \cdot \widehat{g}
	\quad \text{pointwise}.
\end{equation}

\section{Smoothing}

We can write the cumulative distribution function of a random variable $X$ as
$$
\Prob{X \le a} = \E \one_{(-\infty,0]} (X-a),
$$
where $\one_{(-\infty,0]}(x)$ is the indicator of $(-\infty,0]$. We now smooth this indicator, replacing it by a function that transitions gradually from $1$ to $0$, with most of the transition occurring in an $\e$-neighborhood of the origin.

\begin{lemma}[Smoothing the discrepancy]	\label{fourier: lem: smoothing discrepancy}
	Let $\varphi$ be a probability density function, which is also a Schwartz function. Consider a smoothed version of the indicator function of $(-\infty,0]$:
	$$
	f = \one_{(-\infty,0]} \ast \varphi.
	$$
	Let $X$ and $Y$ be random variables. Assume that the probability density function of $Y$ is bounded by $M$. Then we have for any $\e>0$:
	\begin{equation}	\label{fourier: eq: smoothing discrepancy}
		\sup_{a \in \R} \abs[\Big]{\Prob{X \le a} - \Prob{Y \le a}}
		\le 2\sup_{a \in \R} \abs*{\E f \Big(\frac{X-a}{\e}\Big) - \E f \Big(\frac{Y-a}{\e}\Big)} + C_\varphi M \e.
	\end{equation}
	Here $C_\varphi$ depends only on the choice of the smoothing function $\varphi$.  
\end{lemma}

\begin{proof}
	{\em Step 1. Regularity of the discrepancy function.} 
	By rescaling, it is enough to consider $\e=1$. Our task is to bound the discrepancy function 
	$$
	\Delta(a) \coloneqq F_X(a) - F_Y(a),
	$$
	where $F_X(a) \coloneqq \Prob{X \le a}$ and $F_Y(a) \coloneqq \Prob{Y \le a}$ are the cumulative distribution functions. 
	
	The function $F_X$ increases, while $F_Y$ cannot increase too fast because its derivative is bounded by $M$. Putting this together, we see that $\Delta$ cannot decrease too fast: 
	\begin{equation}	\label{fourier: eq: Delta+}
		\Delta(\bar{a}+t) 
		\ge \Delta(\bar{a}) - Mt
		\quad \text{for any } \bar{a} \in \R, \; t \ge 0.
	\end{equation}
	In particular, once $\Delta$ achieves its maximum, $\Delta$ will have to remain close to its maximum for a while. To make this precise, assume without loss of generality that
	\begin{equation}	\label{fourier: eq: Delta max}
		\bar{\Delta} 
		\coloneqq \sup_{a \in \R} \abs*{\Delta(a)}
		= \sup_{a \in \R} \Delta(a).
	\end{equation}
	(Otherwise replace $X, Y$ by $-X, -Y$.) Choose a point $\bar{a}$ where $\Delta(\bar{a}) \ge 0.9 \bar{\Delta}$. Then \eqref{fourier: eq: Delta+} gives 
	$$
	\Delta(\bar{a}+t) 
	\ge 0.9 \bar{\Delta} - Mt 
	\ge \bar{\Delta}/2
	$$
	as long as $0 \le t \le \bar{\Delta}/3M$. In other words, we found an interval of length $\bar{\Delta}/3M$ on which $\Delta$ is bounded below by $\bar{\Delta}/2$.
	
	\smallskip
		
	{\em Step 2. The smoothed discrepancy function.}
	Our task is to compare the discrepancy function $\Delta(a)$ to its smoothed version 
	$$
	\Delta_f(a) 
	\coloneqq \E f(X-a) - \E f(Y-a).
	$$
	To express $\Delta_f$ in terms of $\Delta$, note that 	
	$$
	\E f(X-a) 
= \E \int_\R \one_{(-\infty,0]}(X-a-y) \varphi(y) \, dy
	= \int_\R F_X(a+y) \varphi(y) \, dy
	$$
	by the Fubini theorem. Express $\E f(Y-a)$ similarly, subtract, and obtain 
	\begin{equation}	\label{fourier: eq: Delta f}
		\Delta_f(a) 
		= \int_{-\infty}^\infty \Delta(a+y) \varphi(y) \, dy.
	\end{equation}
	
	\smallskip

	{\em Step 3. Integrating.} 
	In Step 1, we found an interval $[a_0-T, a_0+T]$ with $T = \bar{\Delta}/6M$ and on which $\Delta$ is bounded below by $\bar{\Delta}/2$. Let's decompose \eqref{fourier: eq: Delta f}:
	$$
	\Delta_f(a_0) 
	= \underbrace{\int_{\abs{y} \le T} \Delta(a_0+y) \varphi(y) \, dy}_{I_1}
		+ \underbrace{\int_{\abs{y} > T} \Delta(a_0+y) \varphi(y) \, dy}_{I_2}.
	$$
	To bound $I_2$, recall that \eqref{fourier: eq: Delta max} implies that $\Delta$ is bounded below by $-\bar{\Delta}$ everywhere. Moreover, $\varphi$ is a Schwartz function, so $\int_{\abs{y} > T} \varphi(y) \, dy \le C_\varphi / T$.
	Thus
	$$
	I_2 
	\ge -\bar{\Delta} \cdot \frac{C_\varphi}{T}.
	$$
	To bound $I_1$, recall that $\Delta$ is bounded below by $\bar{\Delta}/2$ in the range of integration. Moreover, $\varphi$ is a probability density function, so its total integral equals $1$. Thus 
	$$
	I_1
	\ge \frac{\bar{\Delta}}{2} \cdot \Big(1 - \frac{C_\varphi}{T} \Big).
	$$
	Adding the two bounds, we conclude that 
	$$
	\Delta_f(a_0)
	\ge \frac{\bar{\Delta}}{2} - \frac{3C_\varphi \bar{\Delta}}{2T}
	\ge \frac{\bar{\Delta}}{2} - 9C_\varphi M.
	$$
	Rearranging the terms yields
	$\bar{\Delta} \le 2\Delta_f(a_0) + 18C_\varphi M$,
	which completes the proof. 
\end{proof}

\medskip

In Lemma~\ref{fourier: lem: smoothing discrepancy}, we replaced the indicator function $\one_{(-\infty,0]}$ by a smooth function $f$. Now, we further replace $f$ with a very particular choice: the complex exponential. 

\begin{lemma}[Smoothed discrepancy via characteristic functions]	\label{fourier: lem: smoothed discrepancy chf}
	There exists a probability density function $\varphi$, which is also a Schwartz function, with the following property. Consider a smoothed version of the indicator function of $(-\infty,0]$:
	$$
f = \one_{(-\infty,0]} \ast \varphi.
	$$
	Let $X$ and $Y$ be random variables. Then we have for any $\e>0$:
	\[
	\sup_{a \in \R} \abs*{\E f \Big(\frac{X-a}{\e}\Big) - \E f \Big(\frac{Y-a}{\e}\Big)} 
	\le \int_{-1/\e}^{1/\e} \abs*{\frac{\E e^{itX} - \E e^{itY}}{t}} \, dt.
	\]
\end{lemma}

\begin{proof}
	{\em Step 1. Choosing a smoothing function.}
	By translation and dilation, we can assume that $a=0$ and $\e=1$, so it suffices to prove the following version of the conclusion:
	$$
	\abs*{\E f(X) - \E f(Y)}
	\le \int_{-1}^1 \abs*{\frac{\E e^{itX} - \E e^{itY}}{t}} \, dt.
	$$
	Choose any probability density function $\varphi$, which is also a Schwartz function, and whose Fourier transform is supported in the interval $\big[-1/(2\pi),1/(2\pi)\big]$. 
	
	(Why does such $\varphi$ exist? Take any even Schwartz function $\psi$ supported in $[-1/4\pi, 1/4\pi]$ and satisfying $\int_\R \psi(x)^2 \, dx = 1$, and set 
	$$
	\varphi(t) \coloneqq \widehat{\psi}(t)^2.
	$$
	Since $\psi$ is even and Schwartz, $\widehat{\psi}$ is real-valued and Schwartz, so $\varphi$ is nonnegative and Schwartz. Moreover, Plancherel identity \eqref{fourier: eq: Plancherel} gives $\int_\R \widehat{\psi}(t)^2 \, dt = \int_\R \psi(x)^2 \, dx = 1$, showing that $\varphi$ is a probability density function. 
Finally, we have $\widehat{\varphi} = \psi \ast \psi$: to see this, first apply the convolution theorem \eqref{fourier: eq: convolution theorem} and then the Fourier inversion formula \eqref{fourier: eq: Fourier inversion}, noting that $\psi$ is even. Thus, $\widehat{\varphi}$ is supported on $\big[-1/(2\pi),1/(2\pi)\big]$.)
	
	\smallskip
	
	{\em Step 2. Bounding the Fourier transform.}
	Since the function $f$ approaches $1$ at $-\infty$, it is not integrable. To fix this, let us consider the integrable function 
	$$
	f_M = \one_{(-M,0]} \ast \varphi.
	$$
	Since the indicators $\one_{(-M,0]}$ increase to $\one_{(-\infty,0]}$ pointwise as $M \to \infty$, the monotone convergence theorem implies that the functions $f_M$ increase to $f$ pointwise. Another application of the monotone convergence theorem yields
	\begin{equation}	\label{fourier: eq: fM vs f}
		\E f_M(X) \to \E f(X), \quad \E f_M(Y) \to \E f(Y)
		\quad \text{as } M \to \infty.
	\end{equation}
	We claim that 
	\begin{equation}	\label{fourier: eq: fM hat}
		\text{$\widehat{f_M}$ is supported on $\big[ -\frac{1}{2\pi}, \frac{1}{2\pi} \big]$} 
		\quad \text{and} \quad
		\abs[\big]{\widehat{f_M}(t)} \le \frac{1}{\pi \abs{t}}, \quad t \in \R.
	\end{equation}
	Indeed, the convolution theorem \eqref{fourier: eq: convolution theorem} shows that $\widehat{f_M} = \widehat{\one_{(-M,0]}} \cdot \widehat{\varphi}$. Now, 
	$$
	\widehat{\one_{(-M,0]}}(t) 
	= \int_{-M}^0 e^{-2 \pi i t x} \, dx
	= \frac{e^{2 \pi i t M} - 1}{2 \pi i t},
	\quad \text{so} \quad
	\abs*{\widehat{\one_{(-M,0]}}(t)}
	\le \frac{1}{\pi \abs{t}}.
	$$
	Moreover, by construction, $\abs*{\widehat{\varphi}(t)}$ vanishes outside $\big[-1/(2\pi),1/(2\pi)\big]$. It is bounded by $\int_\R \abs{\varphi(x)} \, dx = 1$ since $\varphi$ is a probability density function. Combining these bounds proves the claim.
	
	\smallskip
	
	{\em Step 3. Integrating.}
	Using the inverse Fourier transform formula, we can write
	$$
	\E f_M(X)
	= \E \int_\R e^{2 \pi i t X} \widehat{f_M}(t) \, dt
	= \int_\R \E \big[ e^{2 \pi i t X} \big] \widehat{f_M}(t) \, dt
	$$
	by Fubini theorem (note that $f_M$ is Schwartz by construction, so $\widehat{f_M}$ is Schwartz, too). Write the same for $\E f_M(Y)$, subtract and use Jensen's inequality to get
	\begin{align*}
		\abs*{\E f_M(X) - \E f_M(Y)}
		&\le \int_\R \abs*{\E e^{2 \pi i t X} - \E e^{2 \pi i t Y}} \abs[\big]{\widehat{f_M}(t)} \, dt \\
		&\le \int_{-1/(2\pi)}^{1/(2\pi)} \abs*{\E e^{2 \pi i t X} - \E e^{2 \pi i t Y}} \cdot \frac{1}{\pi \abs{t}} \, dt
			\quad \text{(using \eqref{fourier: eq: fM hat}).}
	\end{align*}
	Make the change of variable $s = 2 \pi t$, take limit as $M \to \infty$ using \eqref{fourier: eq: fM vs f}, and the proof is complete. 
\end{proof}

Combining Lemmas \ref{fourier: lem: smoothing discrepancy} and \ref{fourier: lem: smoothed discrepancy chf}, we immediately obtain 

\begin{theorem}[Esseen's smoothing inequality \cite{Esseen smoothing}]\label{fourier: thm: smoothing inequality}
	Let $X$ and $Y$ be random variables. Assume that the probability density function of $Y$ is bounded by $M$. Then we have for any $\e>0$:
	\[
	\sup_{a \in \R} \abs[\Big]{\Prob{X \le a} - \Prob{Y \le a}}
	\le 2 \int_{-1/\e}^{1/\e} \abs*{\frac{\E e^{itX} - \E e^{itY}}{t}} \, dt 
		+ C M \e.
	\]
	Here $C$ is an absolute constant. 
\end{theorem}

\section{Characteristic functions}

Esseen's smoothing inequality (Theorem~\ref{fourier: thm: smoothing inequality}) reduces the problem of approximating a cumulative distribution function $\Prob{X \le a}$ to approximating the characteristic function $\E e^{itX}$. So, what can we say about the characteristic function?

\begin{lemma}[The characteristic function of a random variable] \label{fourier: lem: chf rv}
	Let $X$ be a random variable with 
	\begin{equation}	\label{fourier: eq: X X2 X3}
		\E X = 0, \quad 
		\E X^2 = \s^2, \quad
		\E \abs{X}^3 = \rho^3 < \infty.
	\end{equation}
	Then
	\begin{equation}	\label{fourier: eq: chf approximation}
		\E e^{itX}
		= \exp \Big( -\frac{\s^2 t^2}{2} + O(\rho^3 t^3) \Big)
		\quad \text{whenever} \quad \abs{t} \le \frac{1}{\rho}
	\end{equation}
	and 
	\begin{equation}	\label{fourier: eq: chf bound}
		\abs*{\E e^{itX}}
		\le \exp \Big( -\frac{\s^2 t^2}{2} + O(\rho^3 t^3) \Big)
		\quad \text{for any} \quad t \in \R.
	\end{equation}
\end{lemma}

In the statement and proof of this lemma, we use the $O(\cdot)$ notation to hide factors that are bounded by absolute constants. Precisely, $O(a)$ stands for $\theta a$ where $\theta$ is some quantity that satisfies $\abs{\theta} \le C$, where $C$ is an absolute constant. The quantity $\theta$ and the constant $C$ may change from line to line. 

\begin{proof}
	{\em Step 1. Approximating the exponential function.}
	To prove \eqref{fourier: eq: chf approximation}, write a Taylor approximation of the exponential function:
	$$
	e^{ix} 
	= 1 + ix - \frac{x^2}{2} + \theta_0 x^3
	\quad \text{for some } \theta_0 = \theta_0(x) 
	\text{ satisfying } \abs{\theta_0} \le \frac{1}{6}.
	$$
	(This holds since the third derivative of $e^{ix}$ is bounded by $1$ in modulus.)
	Substitute $x=tX$ and take expectation to get
	$$
	\E e^{itX}
	= 1 + it \E X - \frac{t^2 \E X^2}{2} + t^3 \E [\theta_0 X^3]
	$$
	for some random variable $\theta_0$ satisfying $\abs{\theta_0} \le \frac{1}{6}$ pointwise. Using the assumptions \eqref{fourier: eq: X X2 X3}, we get 
	\begin{equation}	\label{fourier: eq: chf quadratic}
		\E e^{itX}
		= 1 - \frac{\s^2 t^2 }{2} + \theta \rho^3 t^3
		\quad \text{for some } \theta = \theta(x) 
		\text{ satisfying } \abs{\theta} \le \frac{1}{6}.
	\end{equation}
	For convenience, let us rewrite this as 
	\begin{equation}	\label{fourier: eq: chf a b}
		\E e^{itX}
		= 1 - \frac{a}{2} + \theta b, 
		\quad \text{where} \quad 
		a = \s^2 t^2 \text{ and }
		b = \rho^3 t^3,
	\end{equation}
	and note that
	\begin{equation}	\label{fourier: eq: a2 b 1}
		a^2 \le \abs{b} \le 1.
	\end{equation}
	(The second inequality follows from the assumption $\abs{t} \le \frac{1}{\rho}$. To check the first inequality, note that $\s \le \rho$ by Jensen's inequality, so $a^2 = \s^4 t^4 \le \rho^4 t^4 = \abs{b}^{4/3} \le \abs{b}$.)
	
	\smallskip
	
	{\em Step 2. Linearizing the logarithmic function.} Now write a Taylor approximation of the logarithmic function:
	$$
	\ln(1+x) = x + O(x^2)
	\quad \text{whenever} \quad
	\abs{x} \le \frac{2}{3}.
	$$
	We can use this for $x \coloneqq -\frac{a}{2} + \theta b$, since \eqref{fourier: eq: a2 b 1} and \eqref{fourier: eq: chf quadratic} guarantee that $\abs{x} \le \frac{1}{2} + \frac{1}{6} = \frac{2}{3}$. We get
	$$
	\ln \Big( 1 - \frac{a}{2} + \theta b \Big)
	= - \frac{a}{2} + \theta b + O \Big( \big(- \frac{a}{2} + \theta b \big)^2 \Big).
	$$
	Expanding the square and letting the $O(\cdot)$ notation absorb the factors bounded by absolute constants, we conclude that
	$$
	\ln \Big( 1 - \frac{a}{2} + \theta b \Big)
	= - \frac{a}{2} + O(b) + O(a^2) + O(ab) + O(b^2)
	= - \frac{a}{2} + O(b),
	$$
	where the last step follows from \eqref{fourier: eq: a2 b 1}.
	Recalling \eqref{fourier: eq: chf a b}, we see that we proved that
	$$
	\ln \left( \E e^{itX} \right) = - \frac{\s^2 t^2}{2} + O(\rho^3 t^3),
	$$
	as claimed. 
	
	\smallskip
	
	{\em Step 3. Deducing the bound \eqref{fourier: eq: chf bound}.} 
	In the range $\abs{t} \le \frac{1}{\rho}$, the bound \eqref{fourier: eq: chf bound} follows from \eqref{fourier: eq: chf approximation}. And if $\abs{t} > \frac{1}{\rho}$, the bound is nearly trivial. Indeed, since $\abs{e^{ix}} = 1$ holds pointwise, we have
	$$
	\abs{\E e^{itX}} \le 1.
	$$ 
	On the other hand, $\s \le \rho$ and $\rho \abs{t} > 1$ yield $\s^2 t^2/2 \le \rho^2 t^2/2 \le \rho^3 \abs{t}^3$, so 
	$$
	\exp \Big( -\frac{\s^2 t^2}{2} + \rho^3 \abs{t}^3 \Big)
	\ge \exp(0) = 1,
	$$
	and \eqref{fourier: eq: chf bound} follows.
\end{proof}

\smallskip

\begin{remark}[No approximation everywhere]
	 One might ask whether the bound \eqref{fourier: eq: chf approximation} could hold for all $t \in \R$, in which case we won't need the separate bound \eqref{fourier: eq: chf bound}. This is false in general. For instance, a characteristic function $\E e^{itX}$ may have compact support; then the left-hand side of \eqref{fourier: eq: chf approximation} vanishes for large $\abs{t}$, while the right-hand side remains positive.
\end{remark}

Our next goal is to approximate the characteristic function of a sum of independent random variables. To do this, we will multiply the bounds from Lemma~\ref{fourier: lem: chf rv} to obtain:

\begin{lemma}[The characteristic function of a sum]	\label{fourier: lem: chf sum}
	There exist absolute constants $C,c>0$ so that the following holds. 
	Let $X_1,\ldots,X_n$ be independent random variables that satisfy
	$$
	\E X_k = 0, \quad 
	\E X_k^2 = \s_k^2, \quad
	\E \abs{X_k}^3 = \rho_k^3 < \infty.
	$$
	Assume that
	$$
	\sum_{k=1}^n \s_k^2 = 1
	\quad \text{and let} \quad
	\rho^3 \coloneqq \sum_{k=1}^n \rho_k^3.
	$$
	Then the sum $S_n \coloneqq X_1 + \cdots + X_n$ satisfies
	$$
	\abs*{\E e^{itS_n} - e^{-t^2/2}}
	\le C \rho^3 \abs{t}^3 e^{-t^2/4}
	\quad \text{whenever} \quad \abs{t} \le \frac{c}{\rho^3}.
	$$
\end{lemma}

\begin{proof}
	{\em Step 1. Assume that $\abs{t} \le \frac{1}{\rho}$.}
	Then $\abs{t} \le \frac{1}{\rho_k}$ for each $k$, which allows us to apply \eqref{fourier: eq: chf approximation} and get 
	$$
	\E e^{itX_k}
	= \exp \Big( -\frac{\s_k^2 t^2}{2} + \theta_k \rho_k^3 t^3 \Big)
	\quad \text{for some } \theta_k = \theta_k(t) 
	\text{ satisfying } \abs{\theta_k} \le C.
	$$
	By independence, this yields
	$$
	\E e^{itS_n} 
	= \prod_{k=1}^n \E e^{itX_k}
	= \exp \Big( -\frac{t^2}{2} + \theta \rho^3 t^3 \Big)
	\quad \text{for some } \theta = \theta(t) 
	\text{ satisfying } \abs{\theta} \le C.
	$$
	Therefore
	$$
	\abs*{\E e^{itS_n} - e^{-t^2/2}}
	= e^{-t^2/2} \abs*{e^{\theta \rho^3 t^3}-1}
	\le C_1 \rho^3 \abs{t}^3 e^{-t^2/2},
	$$
	and we are done. (The last bound follows once we apply the Taylor approximation of the exponential function $\abs{e^x-1} \le \abs{x} e^{\abs{x}}$ for $x \coloneqq \theta \rho^3 t^3$ and note that $\abs{x} \le C$ by assumption on $t$.) 
	
	\smallskip
	
	{\em Step 2. Assume that $\frac{1}{\rho} < \abs{t} \le \frac{c}{\rho^3}$.}
	Arguing similarly to Step~1, but applying \eqref{fourier: eq: chf bound} instead, we obtain 
	$$
	\abs*{\E e^{itS_n}}
	\le \exp \Big( -\frac{t^2}{2} + \theta \rho^3 t^3 \Big)
	\quad \text{for some } \theta = \theta(t) 
	\text{ satisfying } \abs{\theta} \le C.
	$$
	Choosing the absolute constant $c>0$ in the assumption on $t$ small enough, we can make sure that $\theta \rho^3 t^3 \le t^2/4$. This gives 
	$$
	\abs*{\E e^{itS_n}} 
	\le e^{-t^2/4}.
	$$
	Hence, by triangle inequality, we conclude that
	$$
	\abs*{\E e^{itS_n} - e^{-t^2/2}}
	\le e^{-t^2/4} + e^{-t^2/2}
	\le 2 \rho^3 \abs{t}^3 e^{-t^2/4},
	$$
	since $\rho^3 \abs{t}^3 \ge 1$ by assumption. The lemma is proved. 
\end{proof}

\section{Proof of Theorem~\ref{thm: berry-esseen}}

Now we are ready to prove the Berry--Esseen theorem, Theorem~\ref{thm: berry-esseen}. Set
$$
\rho^3 
\coloneqq \sum_{k=1}^{n} \E \abs{X_{k}}^{3}.
$$
Apply Esseen's smoothing inequality (Theorem~\ref{fourier: thm: smoothing inequality}) with $X=S_n$, $Y=G$, and $1/\e=c/\rho^3$. Since the density of $G$ is bounded by an absolute constant and its characteristic function equals $e^{-t^2/2}$, we obtain
$$
\sup_{a \in \R} \abs[\Big]{\Prob{S_n \le a} - \Prob{G \le a}}
\le 2 \int_{-c/\rho^3}^{c/\rho^3} \abs*{\frac{\E e^{itS_n} - e^{-t^2/2}}{t}} \, dt 
	+ C_1 \rho^3.
$$
Now substitute the bound on the characteristic function of $S_n$ given by Lemma~\ref{fourier: lem: chf sum}. We get
$$
\sup_{a \in \R} \abs[\Big]{\Prob{S_n \le a} - \Prob{G \le a}}
	\le C_2 \rho^3 \int_{-\infty}^\infty t^2 e^{-t^2/4} \, dt + C_1 \rho^3
	\le C_3 \rho^3.
$$
The Berry--Esseen theorem is proved. \qed

\part{The proof by Stein's method}

This part is also self-contained. It restarts from the proof strategy and develops its own smoothing lemma, Stein equation, and inductive argument. Nothing from Part~I is used.

\section{The proof strategy}	\label{stein: s: proof strategy}

We begin with a bird's-eye, non-rigorous sketch of the proof of Theorem~\ref{thm: berry-esseen}.

\smallskip 

\subsection{Smoothing.}

We begin by writing the cumulative distribution function of $S_n$ as 
\begin{equation}	\label{stein: eq: prob Sn}
	\Prob{S_{n} \le a} = \E \one_{\{ S_{n} \le a \}}.
\end{equation}
The indicator function \(x \mapsto \one_{\{x \le a\}}\) is difficult to handle analytically because it jumps from $1$ to $0$ at $a$. To soften this jump, fix $\e>0$ and define the following {\em $\e$-smoothing function at $a$}:
\begin{equation}	\label{stein: eq: smoothing function}
	\begin{gathered}
		h_a(x)
		\coloneqq
		\begin{cases}
			1, & x\le a,\\
			1-\dfrac{x-a}{\e}, & a<x<a+\e,\\
			0, & x\ge a+\e.
		\end{cases}
	\end{gathered}
	\qquad \qquad
	\begin{tikzpicture}[
		baseline=(current bounding box.center),
		x=0.8cm,
		y=1.2cm
	]
		\draw[->] (-0.3,0) -- (4.6,0);
		\draw[very thick] (-0.2,1) -- (1.5,1) -- (2.5,0) -- (4.3,0);
		\draw[dashed] (1.5,0) -- (1.5,1);
		\node[below, text height=1.5ex, text depth=0.25ex] at (1.5,0) {$a$};
		\node[below, text height=1.5ex, text depth=0.25ex] at (2.7,0) {$a+\e$};
		\node[left] at (-0.2,1) {$1$};
	\end{tikzpicture}
\end{equation}
This function decreases linearly from $1$ to $0$ over the interval $[a,a+\e]$. With this smoothing function, \eqref{stein: eq: prob Sn} becomes
\[
\Prob{S_{n} \le a} \approx \E h_a(S_{n}).
\]
Lemma~\ref{stein: lem: smoothing} makes this approximation rigorous and records the error introduced by smoothing.

\subsection{Gaussian integration by parts}

Let $\varphi(x)=(2\pi)^{-1/2}e^{-x^2/2}$ denote the standard normal density. Since $\varphi'(x)=-x\varphi(x)$, integration by parts gives
$$
	\E f'(G)
	= \int_{\R} f'(x)\varphi(x)\,dx
	= \int_{\R} xf(x)\varphi(x)\,dx
	= \E Gf(G)
$$
for every continuously differentiable function $f$ such that $f$ and $f'$ are bounded. Thus, if $S_n$ were Gaussian, the difference
\begin{equation}	\label{stein: eq: Stein difference}
	\E f'(S_n)-\E S_nf(S_n)
\end{equation}
would vanish. Stein's method turns this observation around: if \eqref{stein: eq: Stein difference} is small for suitable functions $f$, then $S_n$ is approximately Gaussian.

\subsection{Stein's ODE}

Consider the following ordinary differential equation, which we call Stein's ODE:
\begin{equation}	\label{stein: eq: Stein ODE}
	f'(x) - x f(x)
	= h(x) - \E h(G),
	\qquad \text{where } G \sim N(0,1),
\end{equation}
where $h=h_a$ is the $\e$-smoothing function \eqref{stein: eq: smoothing function}. Suppose for the moment that this ODE has a well-behaved solution $f$; we will verify this in Sections~\ref{stein: s: solving Stein ODE} and \ref{stein: s: properties of solution}. Substitute $x=S_n$ and take expectations:
\begin{equation}	\label{stein: eq: Stein ODE consequence}
	\E f'(S_n) - \E S_n f(S_n)
	= \E h(S_n) - \E h(G).
\end{equation}
Together with smoothing, this identity suggests that
$$
\abs[\Big]{\Prob{S_{n} \le a} - \Prob{G \le a}}
\approx \abs[\Big]{\E h(S_n) - \E h(G)}
= \abs[\Big]{\E S_n f(S_n) - \E f'(S_n)}.
$$
Our new task is therefore to show that
\begin{equation}	\label{stein: eq: small?}
	\E S_n f(S_n) \approx \E f'(S_n).
\end{equation}
We have reduced the problem to estimating a quantity intrinsic to the sum $S_n$.

\subsection{Expansion}

Let us unpack the key quantity
\begin{equation}	\label{stein: eq: key quantity}
	\E S_n f(S_n)
	= \sum_{k=1}^n \E X_k f(S_n)
\end{equation}
and estimate its terms separately, beginning with $\E X_n f(S_n)$.
The factors $X_n$ and $f(S_n)$ are not independent because $S_n$ contains $X_n$. Let us isolate this summand:
$$
S_n = S_{n-1} + X_n
$$
The fundamental theorem of calculus gives
$$
f(S_n) = f(S_{n-1}) + X_n \int_0^1 f'(S_{n-1} + t X_n) \, dt.
$$

\subsection{Approximation}
Multiply both sides by $X_n$ and take expectations:
\begin{align}
	\E X_n f(S_n)
	&= \underbrace{\E X_n f(S_{n-1})}_{=0 \textrm{ by independence}}
		+ \E X_n^2 \int_0^1 \underbrace{f'(S_{n-1} + t X_n)}_{\approx f'(S_{n-1}) \textrm{ if $f''$ is controlled}} \, dt \label{stein: eq: FTC} \\
	&\approx \E X_n^2 f'(S_{n-1}) \label{eq: first approx} \\
	&= \s_n^2 \E f'(S_{n-1}) \quad \text{(by independence, denoting $\s_k^2 = \E X_k^2$)} \nonumber \\
	&\approx \s_n^2 \E f'(S_n) \quad \text{(if $f''$ is controlled)}. \label{eq: decoupled}
\end{align}
The same argument applies to every term, giving
$$
\E X_k f(S_n)
\approx \s_k^2 \E f'(S_n),
\quad k=1,\ldots,n.
$$
Summing these estimates and using \eqref{stein: eq: key quantity}, we obtain
$$
\E S_n f(S_n)
\approx \sum_{k=1}^n \s_k^2 \E f'(S_n)
= \E f'(S_n)
$$
because the variances $\s_k^2 = \Var(X_k)$ sum to $1$. This is precisely \eqref{stein: eq: small?}, and it completes the heuristic ``proof''.

\subsection{How to make this rigorous?}

Apart from smoothing, which we handle in Section~\ref{stein: s: smoothing}, the heuristic argument relies on one unverified approximation:
\begin{equation}	\label{stein: eq: f' approximation}
	f'(S_{n-1} + y)
	\approx f'(S_{n-1}).
\end{equation}
We used this approximation in \eqref{stein: eq: FTC} for $y=tX_n$ and in \eqref{eq: decoupled} for $y=X_n$. To verify this approximation, it is enough that $y$ is small (which is intuitively true: $X_n$ is just one term of a large sum $S_n$, so we expect it to be relatively small on average) and that $f''$ is controlled.

Here, $f''$ is the second derivative of a solution to Stein's ODE \eqref{stein: eq: Stein ODE}. We solve the ODE explicitly in Section~\ref{stein: s: solving Stein ODE} and study its solution in Section~\ref{stein: s: properties of solution}. This leads to Lemma~\ref{stein: lem: f' varies slowly}, a rigorous version of \eqref{stein: eq: f' approximation}.

We now carry out these steps.

\section{Smoothing}	\label{stein: s: smoothing}

We can write the cumulative distribution function of a random variable $X$ as
$$
\Prob{X \le a} = \E \one_{(-\infty,a]} (X).
$$
We approximate the indicator function $\one_{(-\infty,a]}$ by the $\e$-smoothing function $h_a$ defined in \eqref{stein: eq: smoothing function}.
For the moment, think of $\e$ as a fixed small quantity. We will eventually choose
$\e = C \sum_{k=1}^{n} \E \abs{X_{k}}^{3}$, which gives the error in the Berry--Esseen theorem. We first record the error introduced by smoothing.

\begin{lemma}[Smoothing]	\label{stein: lem: smoothing}
	Let $h_a$ be the $\e$-smoothing function at $a$, and let $X$ and $Y$ be random variables. Assume that the probability density function of $Y$ is bounded by $M$. Then
	$$\sup_{a \in \R} \abs[\Big]{\Prob{X \le a} - \Prob{Y \le a}}
		\le \sup_{a \in \R} \abs[\Big]{\E h_a(X) - \E h_a(Y)} + M \e.
	$$\end{lemma}

\begin{proof}
	Note the pointwise inequality
	\begin{equation}	\label{stein: eq: h bound}
		\one_{(-\infty,a]}
		\le h_a
		\le \one_{(-\infty,a+\e]}.
	\end{equation}
	Then 
	\begin{align*}
		\Prob{X \le a} - \Prob{Y \le a}
&\le \Prob{X \le a} - \Prob{Y \le a+\e} + M\e \\
		&= \E \one_{(-\infty,a]}(X) - \E \one_{(-\infty,a+\e]}(Y) + M\e \\
		&\le \E h_a(X) - \E h_a(Y) + M\e.
	\end{align*}
	For the reverse bound, apply \eqref{stein: eq: h bound} with $a-\e$ in place of $a$. The same argument gives
	$$
	\Prob{Y \le a} - \Prob{X \le a}
	\le \E h_{a-\e}(Y) - \E h_{a-\e}(X) + M\e.
	$$
	Combining the two bounds completes the proof.
\end{proof}

\section{Solving Stein's ODE}	\label{stein: s: solving Stein ODE}

Fix a bounded continuous function $h: \R \to \R$ and consider Stein's ODE \eqref{stein: eq: Stein ODE}, restated here:
\begin{equation}	\label{stein: eq: Stein ODE restated}
	f'(x) - x f(x)
	= h(x) - \E h(G),
	\qquad \text{where } G \sim N(0,1).
\end{equation}
We solve this ODE by the method of integrating factors. Multiplying both sides by $e^{-x^2/2}$ gives
$$
\left( e^{-x^2/2} f(x) \right)'
= e^{-x^2/2} \Big( h(x) - \E h(G) \Big).
$$
The fundamental theorem of calculus now gives
$$
e^{-x^2/2} f(x)
= \int_{-\infty}^x e^{-y^2/2} \Big( h(y) - \E h(G) \Big) \, dy + C.
$$
Setting $C=0$, we obtain the canonical solution
\begin{equation}	\label{stein: eq: ODE solution 1}
	f(x)
	= e^{x^2/2} \int_{-\infty}^x e^{-y^2/2} \Big( h(y) - \E h(G) \Big) \, dy.
\end{equation}
Any other solution differs from this one by $Ce^{x^2/2}$, so the canonical solution is the unique bounded solution. Moreover, the total integral is zero:
$$
\frac{1}{\sqrt{2\pi}} \int_{-\infty}^\infty e^{-y^2/2} \Big( h(y) - \E h(G) \Big) \, dy
= \E h(G) - \E h(G)
= 0,
$$
and therefore we may rewrite \eqref{stein: eq: ODE solution 1} as
\begin{equation}	\label{stein: eq: ODE solution 2}
	f(x)
	= -e^{x^2/2} \int_x^{\infty} e^{-y^2/2} \Big( h(y) - \E h(G) \Big) \, dy.
\end{equation}

\section{Properties of the solution of Stein's ODE}	\label{stein: s: properties of solution}

The two formulas \eqref{stein: eq: ODE solution 1} and \eqref{stein: eq: ODE solution 2} give three simple bounds on the canonical solution of Stein's ODE.

\begin{lemma}[Bounds on the solution of Stein's ODE] \label{stein: lem: ODE solution properties}
	Let $h: \R \to [0,1]$ be continuous. Then the canonical solution \eqref{stein: eq: ODE solution 1} of Stein's ODE \eqref{stein: eq: Stein ODE restated} satisfies, for every $x \in \R$,
	\begin{equation}	\label{stein: eq: ODE solution properties}
		\abs{f(x)} \le 2; \qquad
		\abs{xf(x)} \le 1; \qquad
		\abs{f'(x)} \le 2.
	\end{equation}
\end{lemma}

\begin{proof}
	First, assume that $x>0$. Since $0 \le h \le 1$ everywhere, \eqref{stein: eq: ODE solution 2} gives
	$$
	\abs{f(x)} 
	\le e^{x^2/2} \int_x^{\infty} e^{-y^2/2} \, dy
	= \int_0^{\infty} e^{-xt-t^2/2} \, dt
	$$
	by the change of variables $y=x+t$. Dropping the term $-xt$ gives
	$$
	\abs{f(x)} \le \int_0^{\infty} e^{-t^2/2} \, dt
	= \sqrt{\frac{\pi}{2}}
	\le 2,
	$$
	proving the first bound in \eqref{stein: eq: ODE solution properties} for $x>0$.
	Dropping the term $-t^2/2$ instead gives
	$$
	\abs{f(x)} \le \int_0^{\infty} e^{-xt} \, dt
	= \frac{1}{x},
	$$
	proving the second bound in \eqref{stein: eq: ODE solution properties} for $x>0$. For $x \le 0$, the same argument applies using \eqref{stein: eq: ODE solution 1} in place of \eqref{stein: eq: ODE solution 2}.
	
	Finally, Stein's ODE \eqref{stein: eq: Stein ODE restated} and the triangle inequality give
	$$
	\abs{f'(x)} 
	\le \abs{x f(x)} + \abs{h(x) - \E h(G)}.
	$$
	The first term on the right is bounded by $1$, as proved above. The second is also bounded by $1$ because the range of $h$ is $[0,1]$. Hence $\abs{f'(x)}\le 2$.
\end{proof}

We return to \eqref{stein: eq: f' approximation}, where we need to show that $f'$ varies slowly. For this, we examine $f''$.

We now specialize to the functions needed in the proof. Let $h=h_a$ be the $\e$-smoothing function at $a$ defined in \eqref{stein: eq: smoothing function}. This function is differentiable except at $a$ and $a+\e$. Differentiating Stein's ODE \eqref{stein: eq: Stein ODE restated} gives
$$
f''(x) = f(x) + xf'(x) + h'(x)
$$
whenever $x\ne a,a+\e$. By \eqref{stein: eq: ODE solution properties}, $\abs{f(x)} \le 2$ and $\abs{f'(x)} \le 2$, while $h'(x) = -\frac{1}{\e} \one_{[a,a+\e]}(x)$ wherever the derivative exists. Thus
\begin{equation}	\label{stein: eq: f''}
	\abs{f''(x)}
	\le 2 + 2\abs{x} + \frac{1}{\e} \one_{[a,a+\e]}(x).
\end{equation}

A bound on $f''$ controls the increments of $f'$. Indeed, the fundamental theorem of calculus gives\footnote{More precisely, the smoothing function $h=h_a$ is Lipschitz and $f$ is continuously differentiable. Hence, by \eqref{stein: eq: Stein ODE restated}, $f'=xf+h-\E h(G)$ is locally absolutely continuous. It follows that $f''=f+xf'+h'$ almost everywhere, and the fundamental theorem of calculus applies to $f'$.}
$$
f'(x+y) - f'(x) 
= \int_x^{x+y} f''(t) \, dt
= y \int_0^1 f''(x+sy) \, ds.
$$
Thus,
$$
\abs{f'(x+y) - f'(x)}
\le \abs{y} \int_0^1 \abs{f''(x+sy)} \, ds.
$$
Substituting $x+sy$ for $x$ in \eqref{stein: eq: f''} and using $\abs{x+sy} \le \abs{x} + \abs{y}$, we obtain
\begin{equation}	\label{stein: eq: derivative varies prelim}
	\abs{f'(x+y) - f'(x)}
	\le \abs{y} \left( 2 + 2\abs{x} + 2\abs{y} + \frac{1}{\e} \int_0^1 \one_{[a,a+\e]}(x+sy) \, ds \right).
\end{equation}
This simplifies to the following deterministic bound.

\begin{lemma}[A deterministic increment bound]	\label{stein: lem: derivative varies}
	Let $h=h_a$ be the $\e$-smoothing function at $a$, defined in \eqref{stein: eq: smoothing function}. Then the canonical solution \eqref{stein: eq: ODE solution 2} of Stein's ODE \eqref{stein: eq: Stein ODE restated} satisfies for all $x,y \in \R$:
	\begin{equation}	\label{stein: eq: derivative varies}
		\abs{f'(x+y) - f'(x)}
		\le 4 \abs{y} \left( 1 + \abs{x} + \frac{1}{\e} \int_0^1 \one_{[a,a+\e]}(x+sy) \, ds \right).
	\end{equation}
\end{lemma}

\begin{proof}
	If $\abs{y} \le 1$, the conclusion follows from \eqref{stein: eq: derivative varies prelim} after replacing $2\abs{y}$ by $2$. If $\abs{y}>1$, the left-hand side of \eqref{stein: eq: derivative varies} is bounded by $4$ by \eqref{stein: eq: ODE solution properties}, while its right-hand side is at least $4\abs{y}>4$.
\end{proof}

\section{Preparing the induction}

We will soon average the deterministic increment bound over the partial sum $S_{n-1}$. But before that, let us introduce the error term in Berry-Esseen theorem we want to bound:

\begin{definition}[The error in the Berry--Esseen theorem]	\label{stein: def: delta n}
	Let $\d_n(\rho)$ denote the infimum of all $\d>0$ such that
	$$
	\sup_{a \in \R} \abs[\Big]{\Prob{S_{n} \le a} - \Prob{G \le a}} \le \d
	$$
	whenever $S_n = X_1+\cdots+X_n$ is a sum of independent random variables satisfying
	\begin{equation}	\label{stein: eq: Xk}
		\E X_k = 0, \qquad  
		\sum_{k=1}^n \E X_k^2 = 1, \qquad
		\sum_{k=1}^n \E \abs{X_k}^3 \le \rho^3.
	\end{equation}
\end{definition}

Our goal is to prove that\footnote{Here and throughout the rest of the proof, the notation $a \lesssim b$ means $a \le Cb$ for an absolute constant $C$.}
\begin{equation}	\label{stein: eq: BE goal}
	\d_n(\rho) \le C \rho^3
\end{equation}
for every $n$ and $\rho$, where $C$ is an absolute constant.

We may restrict ourselves to the case where no single summand accounts for more than half of the total variance:
\begin{equation}	\label{stein: eq: leave-one-out variance}
	\E X_k^2 \le \frac{1}{2}
	\qquad \text{for every } k=1,\ldots,n.
\end{equation}
Indeed, if $\E X_k^2 > \frac{1}{2}$ for some $k$, then
$$
\sum_{j=1}^n \E \abs{X_j}^3 
\ge \E \abs{X_k}^3 
\ge \big(\E X_k^2\big)^{3/2} > 2^{-3/2}
> \frac{1}{3}.
$$
Since $\d_n(\rho) \le 1$ holds trivially, the desired bound \eqref{stein: eq: BE goal} is then immediate as long as we choose $C \ge 3$. Thus, it remains to consider the case \eqref{stein: eq: leave-one-out variance}.

We are now ready for a rigorous version of the key approximation \eqref{stein: eq: f' approximation}.

\begin{lemma}[An averaged increment bound]	\label{stein: lem: f' varies slowly}
	Let $h=h_a$ be the $\e$-smoothing function at $a$ defined in \eqref{stein: eq: smoothing function}, and let $f$ be the canonical solution \eqref{stein: eq: ODE solution 1} of Stein's ODE. Suppose that
	$$
	S_n = X_1+\ldots+X_n, \quad n \ge 2,
	$$ 
	is a sum of independent random variables satisfying \eqref{stein: eq: Xk} and \eqref{stein: eq: leave-one-out variance}. Then, for every fixed $y \in \R$,
	$$
	\E \abs[\Big]{f'(S_{n-1}+y) - f'(S_{n-1})}
	\le C \abs{y} \left( 1 + \frac{\d_{n-1}(2\rho)}{\e} \right).
	$$
\end{lemma}

\begin{proof}
	Apply Lemma~\ref{stein: lem: derivative varies} with $x=S_{n-1}$ and take expectations:
	\begin{equation}	\label{stein: eq: f' varies slowly}
		\E \abs[\Big]{f'(S_{n-1}+y) - f'(S_{n-1})}
		\lesssim \abs{y} \left( 1 + \E \abs{S_{n-1}} + \frac{1}{\e} \int_0^1 \Prob[\big]{S_{n-1}+sy \in [a,a+\e]} \, ds \right)
	\end{equation}
	By Cauchy--Schwarz,
	\begin{equation}	\label{stein: eq: abs moment}
		\E \abs{S_{n-1}}
		\le \sqrt{\Var(S_{n-1})} 
		\le \sqrt{\Var(S_n)}
		= 1. 
	\end{equation}
	It remains to bound $\Prob[\big]{S_{n-1}+sy \in [a,a+\e]}$. Set
	$$
	\s^2 \coloneqq \Var(S_{n-1}) 
	= \sum_{k=1}^{n-1} \E X_k^2 
	\ge \frac{1}{2}
	$$
	where the last inequality follows from \eqref{stein: eq: Xk} and \eqref{stein: eq: leave-one-out variance}. Thus, $S_{n-1}/\s$ is a sum of $n-1$ independent mean-zero random variables $X_k/\s$. Their second moments sum to $1$, and the sum of their third moments is at most $(\rho/\s)^3$. Comparing the distribution functions at the two endpoints of the interval introduces two error terms. Therefore,
	\begin{align*}
		\Prob[\big]{S_{n-1}+sy \in [a,a+\e]}
		&= \Prob[\big]{S_{n-1}/\s \in [b,b+\e/\s]}
			\quad \text{(where $b=(a-sy)/\s$)} \\
		&\le \Prob[\big]{G \in [b,b+\e/\s]} + 2\d_{n-1}(\rho/\s)
			\quad \text{(by Definition~\ref{stein: def: delta n})} \\
		&\lesssim \e + \d_{n-1}(2\rho)
	\end{align*}
	because $\s \ge 1/2$ and the density of $G$ is bounded by an absolute constant. Substituting this estimate and \eqref{stein: eq: abs moment} into \eqref{stein: eq: f' varies slowly} gives
	$$
	\E \abs[\Big]{f'(S_{n-1}+y) - f'(S_{n-1})}
	\lesssim \abs{y} \left( 1 + \frac{1}{\e} \big(\e + \d_{n-1}(2\rho) \big) \right), 
	$$
	which proves the lemma.
\end{proof}

The appearance of $\d_{n-1}$ is the key. We have used the quantity we want to estimate, but with one fewer summand, so induction on $n$ becomes possible.

\section{Proof of Theorem~\ref{thm: berry-esseen}}

We now combine smoothing, the averaged increment bound, and induction on the number of summands. For each $a\in\R$, let $h_a$ be the $\e$-smoothing function at $a$ defined in \eqref{stein: eq: smoothing function}, and let $f_a$ be the corresponding canonical solution \eqref{stein: eq: ODE solution 2} of Stein's ODE. Then
\begin{align}
	\sup_{a \in \R} \abs[\Big]{\Prob{S_{n} \le a} - \Prob{G \le a}}
	&\le \sup_{a \in \R} \abs[\Big]{\E h_a(S_n) - \E h_a(G)} + C \e \nonumber\\
	&= \sup_{a \in \R} \abs[\Big]{\E S_n f_a(S_n) - \E f_a'(S_n)} + C \e. \label{stein: eq: approx Stein}
\end{align}
The first step is Lemma~\ref{stein: lem: smoothing}; the second is Stein's ODE, as in \eqref{stein: eq: Stein ODE consequence}.

Fix $a \in \R$ and suppress the subscript in $f_a$. We again unpack the key quantity
\begin{equation}	\label{stein: eq: unpacking}
	\E S_n f(S_n)
	= \sum_{k=1}^n \E X_k f(S_n).
\end{equation}
Set 
$$
\s_k^2\coloneqq \E X_k^2.
$$ 
Expanding $f(S_n)$ around $S_{n-1}$ as in the heuristic argument \eqref{stein: eq: FTC}--\eqref{eq: decoupled}, we obtain
\begin{equation}	\label{stein: eq: Xn fSn}
	\abs[\Big]{\E X_n f(S_n) - \s_n^2 \E f'(S_n)}
	\le A+B, 
\end{equation}
where
$$
A = \E X_n^2 \int_0^1 \abs[\Big]{f'(S_{n-1} + t X_n) - f'(S_{n-1})} \, dt; \qquad
B = \s_n^2 \E \abs[\Big]{f'(S_n) - f'(S_{n-1})}.
$$
($A$ is the approximation error in \eqref{eq: first approx}, and $B$ is the approximation error in \eqref{eq: decoupled}.)

To bound $A$, condition on $X_n$ and write $\E_n[\cdot]\coloneqq\E[\cdot\mid X_n]$. The law of total expectation gives
$$
A = \E \Big[ X_n^2 \int_0^1 \underbrace{\E_n \abs[\Big]{f'(S_{n-1} + t X_n) - f'(S_{n-1})}}_{\lesssim \abs{X_n} \left( 1 + \d_{n-1}(2\rho)/\e \right) \text{ by Lemma~\ref{stein: lem: f' varies slowly}}} \, dt \Big]
\lesssim \underbrace{\E \abs{X_n}^3}_{\eqqcolon \rho_n^3} \left( 1 + \frac{\d_{n-1}(2\rho)}{\e} \right).
$$
Similarly, 
$$
B = \s_n^2 \E \Big[ \underbrace{\E_n \abs[\Big]{f'(S_{n-1}+X_n) - f'(S_{n-1})}}_{\lesssim \abs{X_n} \left( 1 + \d_{n-1}(2\rho)/\e \right) \text{ by Lemma~\ref{stein: lem: f' varies slowly}}} \Big]
\lesssim \underbrace{\s_n^2 \E \abs{X_n}}_{\le \rho_n^3} \left( 1 + \frac{\d_{n-1}(2\rho)}{\e} \right).
$$
Here we used $\s_n \le \rho_n$ and $\E \abs{X_n} \le \rho_n$. Combining the bounds for $A$ and $B$ controls \eqref{stein: eq: Xn fSn}. The same computation applies to every $k$: relabel the summands so that $X_k$ is last. Thus,
$$
\abs[\Big]{\E X_k f(S_n) - \s_k^2 \E f'(S_n)}
\lesssim \rho_k^3 \left( 1 + \frac{\d_{n-1}(2\rho)}{\e} \right).
$$
Summing these estimates in \eqref{stein: eq: unpacking} and using $\sum_{k=1}^n \s_k^2 = 1$, we obtain
$$
\abs[\Big]{\E S_n f(S_n) - \E f'(S_n)}
\lesssim \rho^3 \left( 1 + \frac{\d_{n-1}(2\rho)}{\e} \right)
\quad \text{for any $\rho$ such that } \sum_{k=1}^n \rho_k^3 \le \rho^3.
$$
This bound is uniform in $a \in \R$. Substituting it into \eqref{stein: eq: approx Stein} and restoring the subscript $a$, we obtain
$$
\sup_{a \in \R} \abs[\Big]{\Prob{S_{n} \le a} - \Prob{G \le a}}
\lesssim \rho^3 \left( 1 + \frac{\d_{n-1}(2\rho)}{\e} \right) + \e.
$$
By Definition~\ref{stein: def: delta n},
$$
\delta_n(\rho) 
\le C \rho^3 \left( 1 + \frac{\d_{n-1}(2\rho)}{\e} \right) + C\e.
$$
Choosing $\e \coloneqq 10C\rho^3$, we get
$$
\delta_n(\rho) 
\le C' \rho^3 + \frac{\d_{n-1}(2\rho)}{10},
$$
where $C' \ge 1$ is an absolute constant. This recurrence closes the induction. Assume that
$$
\d_{n-1}(\rho) \le 5C' \rho^3
\quad \text{for any } \rho.
$$
Then 
$$
\d_n(\rho) 
\le C' \rho^3 + \frac{5C' \cdot 8 \rho^3}{10}
=5C' \rho^3
\quad \text{for any } \rho,
$$
which verifies the induction step. The base case is immediate: $\d_1(\rho) \le 1$, while the assumptions give
$$
\rho^3 \ge \E \abs{X_1}^3 \ge (\E X_1^2)^{3/2} = 1.
$$

The Berry--Esseen theorem is proved. \qed

\section*{Acknowledgements}
I am grateful to Pedro Abdalla and Guangyi Zou, who kindly read an early version of the Fourier-analytic proof and helped eliminate several typos and inaccuracies.

\subsection*{Statement on AI use}
The mathematical arguments and core exposition in this paper are entirely the author's own. After the two notes were completed, ChatGPT Pro was used as an ``artificial referee'' to review the mathematics, polish the writing, and assist in combining the notes.

\end{document}